\documentclass[12pt,reqno]{amsart}

\usepackage[T1]{fontenc}
\usepackage[utf8]{inputenc}
\usepackage{amsmath,amssymb,amsthm,amsfonts}
\usepackage{mathrsfs}
\usepackage{mathtools}
\usepackage{bm}
\usepackage{enumitem}
\usepackage{graphicx}
\usepackage[colorlinks=true,
bookmarks=true,
bookmarksnumbered=true,
bookmarkstype=toc,
linkcolor=blue,
urlcolor=blue,
citecolor=blue]{hyperref}

\theoremstyle{plain}
\newtheorem{theorem}{Theorem}[section]
\newtheorem{lemma}[theorem]{Lemma}
\newtheorem{corollary}[theorem]{Corollary}
\newtheorem{proposition}[theorem]{Proposition}

\theoremstyle{definition}
\newtheorem{definition}[theorem]{Definition}

\theoremstyle{remark}
\newtheorem{remark}[theorem]{Remark}

\makeatletter

\@addtoreset{equation}{section}
\makeatother

\title[Ellipsoidal Fractional Obstacle Problems]{
Ellipsoidal Positivity Sets for Fractional Obstacle Problems with Quadratic Forcing
}

\author[S. Kitano]{Shuhei Kitano}
\address[S. Kitano]{
Research Institute for Science and Engineering, Waseda University,
3-4-1 Okubo, Shinjuku-ku, Tokyo 169-8555, Japan
}
\email{sk.koryo@moegi.waseda.jp}

\thanks{
This work was supported by the Japan Society for the Promotion of Science (JSPS) KAKENHI
Grant-in-Aid for Early-Career Scientists and Grant-in-Aid for JSPS Fellows
[grant number JP25KJ0086].
}

\begin{document}

\begin{abstract}
Let \(n\ge1\), \(0<s<1\), \(c>0\), and let \(A\) be a positive definite symmetric matrix.
We prove that the unique decaying viscosity solution of
\[
\min\{u,\,(-\Delta)^s u-(c-\langle Ax,x\rangle)\}=0
\qquad\text{in }\mathbb R^n
\]
has the form
\[
u(x)=K\max\{1-\langle Bx,x\rangle,0\}^{1+s}
\]
for some \(K>0\) and some positive definite symmetric matrix \(B\).
In particular, its positivity set is an ellipsoid.
For \(s=1/2\) and \(n\ge2\), this result, combined with the classification of Fern\'andez-Real and Yu,
implies that cubic global thin obstacle solutions with nonempty bounded positivity set have
ellipsoidal positivity sets.
This proves the conjecture of Fern\'andez-Real and Yu in the nonempty bounded-positivity case.
\end{abstract}

\maketitle

\section{Introduction}\label{sec:introduction}

Let \(n\ge1\) and \(0<s<1\). We consider the fractional obstacle problem
\begin{equation}\label{eq:obstacle}
\min\{u,\,(-\Delta)^s u-f\}=0
\qquad\text{in }\mathbb R^n.
\end{equation}
Throughout, for a symmetric matrix \(M\), we write \(M>0\) if \(M\) is positive definite.
We consider the quadratic forcing
\begin{equation}\label{eq:forcing}
f(x)=c-\langle Ax,x\rangle,
\qquad c>0,\quad A=A^T>0,
\end{equation}
and the decay condition
\[
u(x)\longrightarrow0
\qquad\text{as }|x|\to\infty.
\]
Here \((-\Delta)^s\) is normalized as the Fourier multiplier with symbol \(|\xi|^{2s}\).
For \(v\in C_c^\infty(\mathbb R^n)\), equivalently,
\[
(-\Delta)^s v(x)
=
c_{n,s}\,\mathrm{P.V.}
\int_{\mathbb R^n}
\frac{v(x)-v(z)}{|x-z|^{n+2s}}\,dz,
\]
where \(c_{n,s}>0\) is the normalization constant and \(\mathrm{P.V.}\) denotes the Cauchy principal value.
We understand \eqref{eq:obstacle} in the viscosity sense.
We use the notation \(t_+:=\max\{t,0\}\).

Our main result gives an explicit description of the solution and its positivity set.

\begin{theorem}[Ellipsoidal solution for quadratic forcing]\label{thm:main}
Let \(n\ge1\), \(0<s<1\), \(c>0\), and let \(A\) be a positive definite symmetric \(n\times n\) matrix.
Then there exist \(K>0\) and a positive definite symmetric matrix \(B\) such that
\[
u(x)=K\bigl(1-\langle Bx,x\rangle\bigr)_+^{1+s}
\]
is the unique decaying viscosity solution of \eqref{eq:obstacle}--\eqref{eq:forcing}.
Moreover, \(B\) satisfies
\[
\frac{A_B}{c_B}=\frac{A}{c},
\]
where \(c_B\) and \(A_B\) are given by \eqref{eq:cB} and \eqref{eq:AB}, respectively.
\end{theorem}

In particular,
\[
\{u>0\}
=
\{x\in\mathbb R^n:\langle Bx,x\rangle<1\},
\]
so the positivity set is an ellipsoid.

The case \(s=1/2\) has a direct application to global solutions of the thin obstacle problem.
We refer to Section~\ref{sec:thin} for the precise formulation of the problem and for the definitions of
the thin space, polynomial growth, and cubic growth used below.
Fern\'andez-Real and Yu \cite{FRY} recently classified global solutions with polynomial growth
and bounded positivity set in terms of a polynomial asymptotic at infinity.
In the cubic case, they proved that the positivity set is star-shaped and conjectured that it is an ellipsoid;
see \cite[Conjecture~4.2]{FRY}.
In the nonempty bounded-positivity case, we show that, after a translation in the thin variables,
their cubic problem reduces to \eqref{eq:obstacle} with \(s=1/2\) and positive definite quadratic forcing.
Theorem~\ref{thm:main} therefore gives the following result.

\begin{theorem}[Cubic global thin obstacle solutions]\label{thm:FRY}
Let \(n\ge2\), and let \(W\) be a global solution of the thin obstacle problem in \(\mathbb R^{n+1}\)
with cubic growth and nonempty bounded positivity set on the thin space.
Then, after a translation in the thin variables, there exist \(K>0\) and a positive definite symmetric
matrix \(B\) such that
\[
\{W(\cdot,0)>0\}
=
\{x\in\mathbb R^n:\langle Bx,x\rangle<1\}
\]
and
\[
W(x,0)
=
K\bigl(1-\langle Bx,x\rangle\bigr)_+^{3/2}.
\]
In particular, the positivity set on the thin space is an ellipsoid.
Hence the conjecture of Fern\'andez-Real and Yu \cite[Conjecture~4.2]{FRY}
holds in the nonempty bounded-positivity case.
\end{theorem}

\begin{remark}[Empty positivity set]\label{rem:empty}
The classification of Fern\'andez-Real and Yu also allows the positivity set of \(W\) on the thin space
to be empty.
In the notation of their classification, if \(q(x,0)\ge0\) for every \(x\), then
\[
W(x,y)=-|y|q(x,y)
\]
is the corresponding global solution, and
\[
\{W(\cdot,0)>0\}=\varnothing.
\]
Thus the nonemptiness assumption in Theorem~\ref{thm:FRY} excludes this empty-positivity case.
\end{remark}

This result complements earlier rigidity results for global solutions of the thin obstacle problem.
Eberle, Ros-Oton and Weiss \cite{ERW} obtained a polynomial-asymptotic classification for global solutions
with compact coincidence set, while Eberle and Yu \cite{EY} proved ellipsoidal rigidity for compact contact
sets of solutions with at most quadratic growth.
Fern\'andez-Real and Yu \cite{FRY} showed that the superquadratic regime allows substantially greater
flexibility for compact contact sets and developed instead a classification of solutions with bounded
positivity set.
Theorem~\ref{thm:FRY} shows that the first superquadratic case, namely cubic growth, still exhibits
ellipsoidal rigidity when the positivity set is nonempty and bounded.

The main ingredient in the proof of Theorem~\ref{thm:main} is the explicit formula of
Abatangelo, Jarohs and Salda\~na \cite{AJS} for fractional Laplacians of powers of ellipsoidal defining functions.
In particular, their formula shows that
\[
\psi_B(x)=\bigl(1-\langle Bx,x\rangle\bigr)_+^{1+s}
\]
has a quadratic fractional Laplacian inside the ellipsoid
\[
E_B=\{x\in\mathbb R^n:\langle Bx,x\rangle<1\}.
\]
Matching this quadratic expression with a prescribed forcing reduces to the surjectivity of a
finite-dimensional coefficient map, which we prove using the Poincar\'e--Miranda theorem.
We then verify the obstacle inequality outside the ellipsoid by comparing the forcing and the fractional
Laplacian along outward normal rays.
Uniqueness of the obstacle problem finally yields the uniqueness of the ellipsoidal profile.

The paper is organized as follows.
In Section~\ref{sec:preliminaries}, we recall the formula of Abatangelo, Jarohs and Salda\~na
and introduce the associated coefficient map.
In Section~\ref{sec:surjectivity}, we prove the surjectivity of this map.
Section~\ref{sec:construction} constructs the ellipsoidal solution and proves its uniqueness,
completing the proof of Theorem~\ref{thm:main} and the bijectivity of the coefficient map.
Section~\ref{sec:thin} establishes the connection with cubic global solutions of the thin obstacle problem
and proves Theorem~\ref{thm:FRY}.
Finally, Section~\ref{sec:beyond} briefly discusses the extension beyond quadratic forcing.

\section{Preliminaries}\label{sec:preliminaries}

We first introduce some notation and recall the basic ingredients used below.

For \(a,b\in\mathbb R^n\), we write \(a\otimes b:=ab^T\).
We denote by \(C(\mathbb R^n)\) the space of continuous functions on \(\mathbb R^n\),
by \(C_b(\mathbb R^n)\) the space of bounded continuous functions on \(\mathbb R^n\),
and by \(C^{1,s}(\mathbb R^n)\) the space of continuously differentiable functions whose first
derivatives are \(s\)-H\"older continuous.
We also set
\[
\mathcal S^n_+
:=
\{B\in\mathbb R^{n\times n}:B=B^T,\ B>0\},
\qquad
\mathbb S^{n-1}
:=
\{x\in\mathbb R^n:|x|=1\}.
\]

Given \(u\in C_b(\mathbb R^n)\), \(x_0\in\mathbb R^n\), \(r>0\), and
\(\phi\in C^2(B_r(x_0))\cap C(\overline{B_r(x_0)})\), define
\[
\begin{aligned}
\mathcal L_r[\phi,u](x_0)
:={}&
c_{n,s}\,\mathrm{P.V.}
\int_{B_r(x_0)}
\frac{\phi(x_0)-\phi(y)}{|x_0-y|^{n+2s}}\,dy\\
&+
c_{n,s}
\int_{\mathbb R^n\setminus B_r(x_0)}
\frac{\phi(x_0)-u(y)}{|x_0-y|^{n+2s}}\,dy.
\end{aligned}
\]
The principal value and the tail integral are finite under these assumptions.

\begin{definition}[Viscosity solutions of the obstacle problem]\label{def:viscosity}
Let \(f\in C(\mathbb R^n)\).

A function \(u\in C_b(\mathbb R^n)\) is a viscosity subsolution of \eqref{eq:obstacle}
if, whenever \(x_0\in\mathbb R^n\), \(r>0\), and
\(\phi\in C^2(B_r(x_0))\cap C(\overline{B_r(x_0)})\) are such that
\[
u(x_0)=\phi(x_0),
\qquad
u\le\phi
\quad\text{in }B_r(x_0),
\]
one has \(u(x_0)\le0\) or
\[
\mathcal L_r[\phi,u](x_0)-f(x_0)\le0.
\]

A function \(u\in C_b(\mathbb R^n)\) is a viscosity supersolution of \eqref{eq:obstacle}
if \(u\ge0\) in \(\mathbb R^n\) and, whenever \(x_0\in\mathbb R^n\), \(r>0\), and
\(\phi\in C^2(B_r(x_0))\cap C(\overline{B_r(x_0)})\) are such that
\[
u(x_0)=\phi(x_0),
\qquad
u\ge\phi
\quad\text{in }B_r(x_0),
\]
one has
\[
\mathcal L_r[\phi,u](x_0)-f(x_0)\ge0.
\]

A function \(u\in C_b(\mathbb R^n)\) is a viscosity solution of \eqref{eq:obstacle}
if it is both a viscosity subsolution and a viscosity supersolution.
\end{definition}

Let \(B\in\mathcal S^n_+\) and set
\[
E_B:=\{x\in\mathbb R^n:\langle Bx,x\rangle<1\},
\qquad
\psi_B(x):=\bigl(1-\langle Bx,x\rangle\bigr)_+^{1+s}.
\]
For \(\theta\in\partial E_B\), define
\[
\mu_{B,s}(d\theta)
:=
\frac{dS(\theta)}
{|\theta|^{n+2s}|B\theta|},
\]
where \(dS\) denotes the Euclidean surface measure on \(\partial E_B\).

\begin{proposition}[Abatangelo--Jarohs--Salda\~na]\label{prop:AJS}
There exists a constant \(\kappa_{n,s}>0\) depending only on \(n\) and \(s\) such that,
for \(x\in E_B\),
\[
(-\Delta)^s\psi_B(x)
=
c_B-\langle A_Bx,x\rangle,
\]
where
\begin{equation}\label{eq:cB}
c_B
=
\kappa_{n,s}
\int_{\partial E_B}\mu_{B,s}(d\theta)
\end{equation}
and
\begin{equation}\label{eq:AB}
A_B
=
\kappa_{n,s}
\left[
B\int_{\partial E_B}\mu_{B,s}(d\theta)
+
2s
\int_{\partial E_B}
(B\theta)\otimes(B\theta)\,
\mu_{B,s}(d\theta)
\right].
\end{equation}
\end{proposition}

For diagonal \(B\), the case \(j=1\) of \cite[Corollary~3.5]{AJS} gives
\[
(-\Delta)^s\psi_B(x)
=
\kappa_{n,s}
\int_{\partial E_B}
\left(
1-\langle Bx,x\rangle
-
2s(x\cdot B\theta)^2
\right)
\mu_{B,s}(d\theta).
\]
Since
\[
(x\cdot B\theta)^2
=
\langle((B\theta)\otimes(B\theta))x,x\rangle,
\]
this is exactly \eqref{eq:cB}--\eqref{eq:AB}.
The general case follows by the orthogonal invariance of the fractional Laplacian.
Notice that \(c_B>0\) and \(A_B\in\mathcal S^n_+\).

We define the coefficient map
\[
F:\mathcal S^n_+\longrightarrow\mathcal S^n_+,
\qquad
F(B):=\frac{A_B}{c_B}.
\]
We shall prove in the following sections that \(F\) is a bijection.

\section{Surjectivity of the coefficient map}\label{sec:surjectivity}

Our aim in this section is to prove the surjectivity of \(F\).
We start with its continuity.

\begin{lemma}[Continuity]\label{lem:continuity}
The coefficient map
\[
F:\mathcal S^n_+\longrightarrow\mathcal S^n_+
\]
is continuous.
\end{lemma}

\begin{proof}
Parametrize \(\partial E_B\) by
\[
\theta=B^{-1/2}\omega,
\qquad
\omega\in\mathbb S^{n-1}.
\]
Under this parametrization,
\[
\mu_{B,s}(d\theta)
=
(\det B)^{-1/2}
|B^{-1/2}\omega|^{-n-2s}\,dS(\omega).
\]
If \(B\) ranges in a compact subset of \(\mathcal S^n_+\), then the integrands obtained from
\eqref{eq:cB} and \eqref{eq:AB} under this parametrization are uniformly bounded on
\(\mathbb S^{n-1}\).
Therefore the dominated convergence theorem shows that these integrals depend continuously on \(B\),
and hence so does \(F(B)\).
\end{proof}

We use the following form of the Poincar\'e--Miranda theorem; see \cite{Kulpa}.

\begin{proposition}[Poincar\'e--Miranda]\label{prop:PM}
Let
\[
Q=\prod_{i=1}^n[\alpha_i,\beta_i],
\qquad
\alpha_i<\beta_i
\quad(i=1,\ldots,n),
\]
and let \(H=(H_1,\ldots,H_n):Q\to\mathbb R^n\) be continuous.
Suppose that, for each \(i\),
\[
H_i\le0
\quad\text{on }\{x\in Q:x_i=\alpha_i\},
\qquad
H_i\ge0
\quad\text{on }\{x\in Q:x_i=\beta_i\}.
\]
Then there exists \(x\in Q\) such that \(H(x)=0\).
\end{proposition}

\begin{lemma}[Surjectivity]\label{lem:surj}
The coefficient map
\[
F:\mathcal S^n_+\longrightarrow\mathcal S^n_+
\]
is surjective.
\end{lemma}

\begin{proof}
We first note that \(F\) is orthogonally equivariant.
Indeed, \eqref{eq:cB} and \eqref{eq:AB} give
\[
F(O^TBO)=O^TF(B)O
\]
for every orthogonal matrix \(O\).
Hence it is enough to consider a diagonal target
\[
M=\operatorname{diag}(m_1,\ldots,m_n),
\qquad
m_i>0.
\]

Let
\[
B=\operatorname{diag}(b_1,\ldots,b_n),
\qquad
b_i>0.
\]
By reflection symmetry, \(A_B\) is diagonal.
Define
\[
J_0(B)
:=
\int_{\partial E_B}\mu_{B,s}(d\theta),
\qquad
J_1^{(i)}(B)
:=
b_i
\int_{\partial E_B}\theta_i^2\,\mu_{B,s}(d\theta),
\]
and
\[
\rho_i(B)
:=
\frac{J_1^{(i)}(B)}{J_0(B)}.
\]
Since
\[
\sum_{i=1}^n b_i\theta_i^2=1
\qquad\text{on }\partial E_B,
\]
we have
\[
\rho_i(B)>0,
\qquad
\sum_{i=1}^n\rho_i(B)=1.
\]
It follows from \eqref{eq:cB}--\eqref{eq:AB} that
\begin{equation}\label{eq:Fdiag}
F(B)
=
\operatorname{diag}
\bigl(
b_1(1+2s\rho_1(B)),\ldots,
b_n(1+2s\rho_n(B))
\bigr).
\end{equation}
In particular, if \(F_i(B)\) denotes the \(i\)-th diagonal entry of \(F(B)\), then
\[
b_i<F_i(B)\le(1+2s)b_i,
\qquad
i=1,\ldots,n.
\]

Identify a positive diagonal matrix \(B\) with
\(b=(b_1,\ldots,b_n)\in(0,\infty)^n\), and consider the box
\[
Q_M
=
\prod_{i=1}^n
\left[
\frac{m_i}{1+2s},
m_i
\right].
\]
Define
\[
H_i(b):=F_i(B)-m_i.
\]
On the lower \(i\)-face, \eqref{eq:Fdiag} gives
\[
F_i(B)\le(1+2s)b_i=m_i,
\]
so \(H_i\le0\).
On the upper \(i\)-face,
\[
F_i(B)>b_i=m_i,
\]
so \(H_i>0\).
Lemma~\ref{lem:continuity} and Proposition~\ref{prop:PM} yield a point
\(b\in Q_M\) for which \(F(B)=M\).
This proves the surjectivity of \(F\).
\end{proof}

\section{Construction and uniqueness of the solution}\label{sec:construction}

We first prove the uniqueness of a decaying viscosity solution of \eqref{eq:obstacle}.

\begin{lemma}[Uniqueness]\label{lem:uniqueness}
The decaying viscosity solution of \eqref{eq:obstacle} is unique.
\end{lemma}

\begin{proof}
Let \(u\) and \(v\) be two decaying viscosity solutions.
They are bounded since they are continuous and decay at infinity.
Suppose that \(u>v\) somewhere.
For \(\varepsilon>0\) sufficiently small, the set
\[
\Omega_\varepsilon
:=
\{u>v+\varepsilon\}
\]
is a nonempty bounded open set.
Since \(v\ge0\), one has \(u>0\) in \(\Omega_\varepsilon\).
Hence, with \(I:=-( -\Delta)^s\),
\[
Iu\ge-f
\qquad\text{in }\Omega_\varepsilon
\]
in the viscosity sense, while the supersolution property of \(v\) gives
\[
I(v+\varepsilon)=Iv\le-f
\qquad\text{in }\Omega_\varepsilon.
\]
Moreover,
\[
u\le v+\varepsilon
\qquad\text{in }\mathbb R^n\setminus\Omega_\varepsilon.
\]
The operator \(I=-(-\Delta)^s\) satisfies the hypotheses of the comparison principle
in \cite[Theorem~5.2]{CS}.
Hence \(u\le v+\varepsilon\) in \(\Omega_\varepsilon\), a contradiction.
Thus \(u\le v\).
Interchanging \(u\) and \(v\) gives \(u=v\).
\end{proof}

\begin{proof}[Proof of Theorem~\ref{thm:main}]
By Lemma~\ref{lem:surj}, there exists \(B\in\mathcal S^n_+\) such that
\[
F(B)=\frac{A}{c}.
\]
Set
\[
K:=\frac{c}{c_B}>0,
\qquad
u_B(x):=K\psi_B(x).
\]
Then Proposition~\ref{prop:AJS} gives
\begin{equation}\label{eq:inside}
(-\Delta)^s u_B(x)
=
c-\langle Ax,x\rangle
=
f(x)
\qquad
\text{for }x\in E_B.
\end{equation}

The profile \(\psi_B\) belongs to \(C^{1,s}(\mathbb R^n)\).
By \cite[Proposition~2.6]{Silvestre},
\[
(-\Delta)^s\psi_B\in C^{0,1-s}(\mathbb R^n).
\]
Hence \eqref{eq:inside} extends continuously to \(\partial E_B\).

It remains to verify the obstacle inequality outside \(E_B\).
Diagonalizing \(B\) and using reflection symmetry in \eqref{eq:AB},
we see that \(A_B\) and \(B\) commute.
Since \(F(B)=A/c\), the matrices \(A\) and \(B\) are therefore simultaneously diagonalizable.
Let \(q\in\partial E_B\) and let
\[
\nu(q):=\frac{Bq}{|Bq|}
\]
be the outer unit normal to \(E_B\) at \(q\).
Put
\[
x_t=q+t\nu(q),
\qquad
t\ge0.
\]
Since \(A\) and \(B\) are simultaneously diagonalizable and positive definite,
\[
\frac{d}{dt}f(x_t)
=
-2\langle A(q+t\nu),\nu\rangle
=
-2\frac{\langle Aq,Bq\rangle}{|Bq|}
-2t\langle A\nu,\nu\rangle
<0.
\]
Thus
\begin{equation}\label{eq:fmon}
f(x_t)\le f(q).
\end{equation}

For \(t>0\), \(x_t\notin\operatorname{supp}u_B\), and hence
\[
(-\Delta)^s u_B(x_t)
=
-c_{n,s}
\int_{E_B}
\frac{u_B(z)}{|x_t-z|^{n+2s}}\,dz.
\]
Set
\[
H(t):=-(-\Delta)^s u_B(x_t)>0.
\]
Differentiating under the integral sign gives
\[
H'(t)
=
-(n+2s)c_{n,s}
\int_{E_B}
u_B(z)
\frac{(q-z)\cdot\nu+t}
{|q+t\nu-z|^{n+2s+2}}\,dz.
\]
Since \(E_B\) is convex and \(\nu(q)\) is a supporting normal,
\[
(q-z)\cdot\nu\ge0
\qquad
\text{for every }z\in E_B.
\]
Therefore \(H'(t)\le0\), and consequently
\begin{equation}\label{eq:fracmon}
(-\Delta)^s u_B(x_t)
\ge
(-\Delta)^s u_B(q),
\qquad
t\ge0,
\end{equation}
where the case \(t=0\) follows from the continuity established above.

Every \(x\in E_B^c\) has a unique Euclidean projection \(q\in\partial E_B\),
and
\[
x=q+t\nu(q)
\]
for some \(t\ge0\).
Combining \eqref{eq:inside}, \eqref{eq:fmon}, and \eqref{eq:fracmon}, we obtain
\[
(-\Delta)^s u_B(x)
\ge
(-\Delta)^s u_B(q)
=
f(q)
\ge
f(x).
\]
Thus
\[
u_B\ge0,
\qquad
(-\Delta)^s u_B\ge f
\quad\text{in }\mathbb R^n,
\qquad
(-\Delta)^s u_B=f
\quad\text{in }\{u_B>0\}=E_B.
\]
Since \((-\Delta)^s u_B\) is continuous, these pointwise inequalities imply the viscosity inequalities
of Definition~\ref{def:viscosity}.
Hence \(u_B\) is a viscosity solution of \eqref{eq:obstacle}.
By Lemma~\ref{lem:uniqueness}, it is the unique decaying viscosity solution.
This proves the theorem.
\end{proof}

\begin{corollary}\label{cor:bijective}
The coefficient map
\[
F:\mathcal S^n_+\longrightarrow\mathcal S^n_+
\]
is bijective.
\end{corollary}

\begin{proof}
Surjectivity is Lemma~\ref{lem:surj}.
Suppose that
\[
F(B_1)=F(B_2)=M.
\]
For \(i=1,2\), set
\[
u_i:=\frac{1}{c_{B_i}}\psi_{B_i}.
\]
The proof of Theorem~\ref{thm:main} shows that both \(u_1\) and \(u_2\) are decaying viscosity solutions
of \eqref{eq:obstacle} with \(c=1\) and \(A=M\).
Lemma~\ref{lem:uniqueness} gives \(u_1=u_2\), and hence \(E_{B_1}=E_{B_2}\).
Since the radial function of \(E_B\) is
\[
\omega\longmapsto\langle B\omega,\omega\rangle^{-1/2}
\qquad
\text{on }\mathbb S^{n-1},
\]
it follows that \(B_1=B_2\).
Thus \(F\) is injective.
In particular, the parameters \(B\) and \(K=c/c_B\) in Theorem~\ref{thm:main}
are uniquely determined.
\end{proof}

\section{Cubic global solutions of the thin obstacle problem}\label{sec:thin}

We now specialize to \(s=1/2\) and prove Theorem~\ref{thm:FRY}.
Throughout this section, \(n\ge2\).

Write points in \(\mathbb R^{n+1}\) as
\[
X=(x,y)\in\mathbb R^n\times\mathbb R,
\]
and call \(\{y=0\}\) the thin space.
We call \(\{W(\cdot,0)>0\}\) the positivity set of \(W\) on the thin space.
Following \cite{FRY}, we say that \(W:\mathbb R^{n+1}\to\mathbb R\) has polynomial growth if
\[
\left\|
\frac{W(X)}{1+|X|^m}
\right\|_{L^\infty(\mathbb R^{n+1})}
<\infty
\]
for some \(m\in\mathbb N\), and call the smallest such \(m\) the order of the polynomial growth of \(W\).
In particular, \(W\) has cubic growth if this order is \(3\).
We also write
\[
\partial_y^+W(x,0)
:=
\lim_{y\downarrow0}\partial_yW(x,y)
\]
whenever the limit exists.
Following \cite{FRY}, a global solution of the thin obstacle problem satisfies
\begin{equation}\label{eq:thin}
\begin{cases}
W\ge0 & \text{on }\{y=0\},\\
\Delta W=0
& \text{in }\mathbb R^{n+1}\setminus\{y=0,\ W=0\},\\
\Delta W\le0
& \text{in }\mathbb R^{n+1},\\
W(x,y)=W(x,-y).
\end{cases}
\end{equation}
Here the inequality \(\Delta W\le0\) is understood in the distributional sense.

Let \(\mathcal P_{n+1}\) denote the space of polynomials on \(\mathbb R^{n+1}\).
Fern\'andez-Real and Yu \cite[Theorem~1.1]{FRY} introduce the class
\[
\mathcal P^o_{n+1}
:=
\left\{
q\in\mathcal P_{n+1}:
\Delta(yq)=0,\ 
q\text{ is even in }y,\ 
\{q(x,0)<0\}\text{ is compact}
\right\}
\]
and prove that a polynomial-growth global solution has bounded positivity set if and only if
\begin{equation}\label{eq:asymptotic}
|W(x,y)+|y|q(x,y)|
\longrightarrow0
\qquad
\text{as }|(x,y)|\to\infty
\end{equation}
for some \(q\in\mathcal P^o_{n+1}\).
For each \(q\), the corresponding global solution is unique.
Their proof also shows that if \(W\) has growth order \(m\), then \(q\) has degree \(m-1\);
see \cite[Section~3]{FRY}.

We first classify the possible cubic data.

\begin{lemma}[Quadratic form of the cubic datum]\label{lem:q}
Let \(W\) have cubic growth and nonempty bounded positivity set, and let
\(q\in\mathcal P^o_{n+1}\) be given by \eqref{eq:asymptotic}.
Then, after a translation in \(x\),
\begin{equation}\label{eq:qform}
q(x,y)
=
\langle Mx,x\rangle
-c
-\frac13(\operatorname{tr}M)y^2,
\end{equation}
where \(M=M^T>0\) and \(c>0\).
\end{lemma}

\begin{proof}
Since \(W\) has cubic growth, \(q\) is quadratic.
Since \(q\) is even in \(y\), it has the form
\[
q(x,y)
=
\langle Mx,x\rangle
+\ell\cdot x
+d
+\gamma y^2,
\qquad
M=M^T.
\]
The condition \(\Delta(yq)=0\) gives
\[
2(\operatorname{tr}M)y+6\gamma y=0,
\]
hence
\[
\gamma=-\frac{\operatorname{tr}M}{3}.
\]

The set \(\{q(x,0)<0\}\) is bounded.
It is also nonempty.
Indeed, if \(q(x,0)\ge0\) for every \(x\), then
\[
W_0(x,y):=-|y|q(x,y)
\]
is a global solution of the thin obstacle problem satisfying \eqref{eq:asymptotic}:
it is harmonic away from the thin space, vanishes on \(\{y=0\}\), and
\[
\Delta W_0
=
-2q(x,0)\,\delta_{\{y=0\}}
\le0
\]
in the distributional sense.
By the uniqueness in \cite[Theorem~1.1]{FRY}, \(W=W_0\),
contradicting the assumption that the positivity set of \(W\) is nonempty.

Boundedness and nonemptiness of the negative set force \(M>0\):
a negative direction of \(M\) would make the negative set unbounded,
while a null direction either produces an unbounded negative half-line through the linear term
or leaves a negative value invariant along an entire line.

Completing the square with
\[
x_0=-\frac12M^{-1}\ell
\]
gives
\[
q(x,y)
=
\langle M(x-x_0),x-x_0\rangle
-c
-\frac13(\operatorname{tr}M)y^2
\]
with
\[
c=\frac14\ell^TM^{-1}\ell-d>0.
\]
Translating \(x_0\) to the origin yields \eqref{eq:qform}.
\end{proof}

The equivalence between the Signorini problem and the obstacle problem for the half-Laplacian
is standard; see, for instance, \cite{CSS} and \cite[Section~1.2]{Silvestre}.
We apply this correspondence after subtracting the polynomial asymptotic in \eqref{eq:asymptotic}.

\begin{proposition}[Reduction to a decaying half-Laplacian obstacle problem]\label{prop:equiv}
Let \(q\) be as in \eqref{eq:qform}.
A global solution \(W\) satisfying \eqref{eq:asymptotic} corresponds bijectively to a decaying
viscosity solution \(g\) of
\begin{equation}\label{eq:half}
\min\left\{
g,\,
(-\Delta)^{1/2}g-\bigl(c-\langle Mx,x\rangle\bigr)
\right\}
=
0
\qquad
\text{in }\mathbb R^n.
\end{equation}
Moreover,
\[
\{W(\cdot,0)>0\}=\{g>0\}.
\]
\end{proposition}

\begin{proof}
Suppose first that \(W\) is given.
By the local optimal regularity for the Signorini problem,
\[
W\in C_{\mathrm{loc}}^{1,1/2}(\mathbb R^{n+1}_+);
\]
see \cite[Proposition~4.5]{FRsurvey}.
In particular, the one-sided normal derivative \(\partial_y^+W(\cdot,0)\) is well defined.

For \(y\ge0\), set
\[
V(x,y):=W(x,y)+yq(x,y).
\]
Since \(\Delta(yq)=0\), the function \(V\) is harmonic in \(\mathbb R^{n+1}_+\),
and \eqref{eq:asymptotic} gives \(V\to0\) at infinity.
Let
\[
g(x):=V(x,0)=W(x,0).
\]
Then
\[
g\in C_{\mathrm{loc}}^{1,1/2}(\mathbb R^n)
\]
and \(g(x)\to0\) as \(|x|\to\infty\).
Hence \(V\) is the decaying harmonic extension of \(g\), and the Dirichlet-to-Neumann
characterization of the half-Laplacian gives pointwise
\[
(-\Delta)^{1/2}g
=
-\partial_y^+V.
\]
Using
\[
q(x,0)=\langle Mx,x\rangle-c,
\]
we obtain
\[
(-\Delta)^{1/2}g
-
\bigl(c-\langle Mx,x\rangle\bigr)
=
-\partial_y^+W.
\]

The local regularity above allows us to read the distributional formulation \eqref{eq:thin}
on the thin space as the Signorini complementarity conditions
\[
W(\cdot,0)\ge0,
\qquad
\partial_y^+W\le0,
\qquad
W(\cdot,0)\,\partial_y^+W=0.
\]
Together with the identity above, they show that \(g\) satisfies \eqref{eq:half},
equivalently in the viscosity sense of Definition~\ref{def:viscosity}.
Moreover,
\[
\{g>0\}
=
\{W(\cdot,0)>0\}.
\]

Conversely, let \(g\) solve \eqref{eq:half}, let \(V\) be its decaying harmonic extension, and define
\[
W(x,y)
:=
V(x,|y|)
-
|y|q(x,y).
\]
By the same Signorini--half-Laplacian correspondence, the viscosity inequalities in \eqref{eq:half}
translate into the Signorini boundary conditions for \(W\); equivalently,
\[
\partial_y^+W
=
-(-\Delta)^{1/2}g
+c-\langle Mx,x\rangle
\le0,
\qquad
g\,\partial_y^+W=0,
\]
in the Dirichlet-to-Neumann sense.
Since \(W\) is harmonic away from the thin space and even in \(y\), it satisfies \eqref{eq:thin}.
Finally,
\[
W+|y|q=V\to0
\]
at infinity, so \eqref{eq:asymptotic} holds.
The two constructions are inverse to one another.
\end{proof}

\begin{proof}[Proof of Theorem~\ref{thm:FRY}]
By Lemma~\ref{lem:q}, after a translation in the thin variables, the polynomial \(q\) has the form
\[
q(x,y)
=
\langle Mx,x\rangle
-c
-\frac13(\operatorname{tr}M)y^2
\]
with \(M>0\) and \(c>0\).
By Proposition~\ref{prop:equiv}, \(g=W(\cdot,0)\) is the decaying viscosity solution of
\eqref{eq:half}.
Theorem~\ref{thm:main} with \(s=1/2\) gives
\[
W(x,0)
=
g(x)
=
K\bigl(1-\langle Bx,x\rangle\bigr)_+^{3/2}
\]
for some \(K>0\) and \(B\in\mathcal S^n_+\).
Moreover, Proposition~\ref{prop:equiv} gives
\[
\{W(\cdot,0)>0\}
=
\{g>0\}
=
\{x\in\mathbb R^n:\langle Bx,x\rangle<1\}.
\]
This proves the theorem.
\end{proof}

\begin{remark}[Weighted extension for general \(s\)]\label{rem:weighted}
For general \(0<s<1\), the Caffarelli--Silvestre extension identifies fractional obstacle problems
with weighted thin obstacle problems for
\[
L_a:=\operatorname{div}(|y|^a\nabla\cdot),
\qquad
a=1-2s;
\]
see \cite{extension} and, for a survey, \cite[Section~3.1]{FRsurvey}.
The restriction to \(s=1/2\) in this section comes from the fact that the classification of
Fern\'andez-Real and Yu used here concerns the classical unweighted thin obstacle problem.
\end{remark}

\section{Beyond quadratic forcing}\label{sec:beyond}

The explicit construction in this paper relies essentially on the quadratic structure of the forcing.
In forthcoming work, we treat the corresponding second-order problem with general concave forcing
and obtain a general power-concavity result.
The situation for nonlocal problems is much less complete.

For the restricted half-Laplacian, Kulczycki \cite{Kulczycki} proved that the torsion function is concave
in bounded convex planar domains.
More recently, Gallo and Squassina \cite[Section~7.3]{GalloSquassina} obtained perturbative interior
concavity estimates for fractional \(p\)-Laplacian eigenfunctions in the regime \(s\uparrow1\).
They also emphasize that very few concavity results are known for fractional equations.
To our knowledge, a general exact power-concavity theory for fractional Poisson equations with
spatially dependent forcing is not available.

It is therefore natural to ask which structural assumptions on \(f\) ensure power concavity of the solution of
\[
(-\Delta)^su=f
\quad\text{in }\Omega,
\qquad
u=0
\quad\text{in }\mathbb R^n\setminus\Omega,
\]
when \(\Omega\) is convex.
For the obstacle problem \eqref{eq:obstacle}, the related question is whether structural assumptions on a
general forcing imply convexity of the positivity set.
The quadratic case treated here provides an explicit model in which both the solution and its positivity
geometry can be described completely.

\section*{Conflict of interest}

The author declares no conflict of interest.

\section*{Funding}

This work was supported by the Japan Society for the Promotion of Science (JSPS)\\
KAKENHI Grant-in-Aid for Early-Career Scientists and Grant-in-Aid for JSPS Fellows
[grant number JP25KJ0086].

\section*{Acknowledgements}

During the preparation of this manuscript, the author used ChatGPT (OpenAI) to assist with language editing,
organization, literature searches, and consistency checks of the mathematical exposition and references.
The author independently verified all mathematical arguments, calculations, citations, and the final text,
and takes full responsibility for the content of the manuscript.


\begin{thebibliography}{99}

\bibitem{AJS}
N.~Abatangelo, S.~Jarohs, and A.~Salda\~na,
\emph{Fractional Laplacians on ellipsoids},
Mathematics in Engineering \textbf{3} (2021), no.~5, 1--34.

\bibitem{CSS}
L.~A.~Caffarelli, S.~Salsa, and L.~Silvestre,
\emph{Regularity estimates for the solution and the free boundary of the obstacle problem for the fractional Laplacian},
Invent. Math. \textbf{171} (2008), no.~2, 425--461.

\bibitem{extension}
L.~Caffarelli and L.~Silvestre,
\emph{An extension problem related to the fractional Laplacian},
Comm. Partial Differential Equations \textbf{32} (2007), 1245--1260.

\bibitem{CS}
L.~Caffarelli and L.~Silvestre,
\emph{Regularity theory for fully nonlinear integro-differential equations},
Comm. Pure Appl. Math. \textbf{62} (2009), 597--638.

\bibitem{ERW}
S.~Eberle, X.~Ros-Oton, and G.~S.~Weiss,
\emph{Characterizing compact coincidence sets in the thin obstacle problem and the obstacle problem for the fractional Laplacian},
Nonlinear Anal. \textbf{211} (2021), 112473.

\bibitem{EY}
S.~Eberle and H.~Yu,
\emph{Compact contact sets of sub-quadratic solutions to the thin obstacle problem},
Adv. Math. \textbf{444} (2024), 109635.

\bibitem{FRsurvey}
X.~Fern\'andez-Real,
\emph{The thin obstacle problem: a survey},
Publ. Mat. \textbf{66} (2022), no.~1, 3--55.

\bibitem{FRY}
X.~Fern\'andez-Real and H.~Yu,
\emph{Global solutions to the thin obstacle problem with superquadratic growth},
arXiv:2504.21492, 2025.

\bibitem{GalloSquassina}
M.~Gallo and M.~Squassina,
\emph{Concavity and perturbed concavity for \(p\)-Laplace equations},
J. Differential Equations \textbf{440} (2025), 113452.

\bibitem{Kulczycki}
T.~Kulczycki,
\emph{On concavity of solution of Dirichlet problem for the equation \((-\Delta)^{1/2}\phi=1\) in a convex planar region},
J. Eur. Math. Soc. \textbf{19} (2017), 1361--1420.

\bibitem{Kulpa}
W.~Kulpa,
\emph{The Poincar\'e--Miranda theorem},
Amer. Math. Monthly \textbf{104} (1997), no.~6, 545--550.

\bibitem{Silvestre}
L.~Silvestre,
\emph{Regularity of the obstacle problem for a fractional power of the Laplace operator},
Comm. Pure Appl. Math. \textbf{60} (2007), no.~1, 67--112.

\end{thebibliography}
\end{document}